\documentclass[11pt,a4paper]{amsart}

\usepackage[T1]{fontenc}
\usepackage[utf8]{inputenc}

\usepackage{commath}
\usepackage{mathtools}
\usepackage{graphicx}
\usepackage{enumitem}
\usepackage{pb-diagram}
\usepackage{amsfonts,amssymb,amsthm,amscd,amsmath}
\usepackage{latexsym,stmaryrd}
\usepackage{lscape}
\usepackage[all]{xy}
\usepackage{setspace}
\usepackage[margin=0.8in]{geometry}
\usepackage{geometry}
\usepackage{bussproofs}
\usepackage[pagebackref=true,colorlinks,linkcolor=blue,citecolor=magenta]{hyperref}
\usepackage[colorlinks,linkcolor=blue,citecolor=magenta]{hyperref}
\usepackage{hyperref}

\theoremstyle{plain}
\newtheorem{theorem}{Theorem}
\newtheorem{lemma}[theorem]{Lemma}
\newtheorem{proposition}[theorem]{Proposition}

\theoremstyle{definition}

\newtheorem{remark}[theorem]{Remark}

\newcommand{\pa}{\partial}

\newcommand{\al}{\alpha}
\newcommand{\be}{\beta}
\newcommand{\e}{\epsilon}
\newcommand{\om}{\omega}
\newcommand{\Om}{\Omega}

\newcommand{\bB}{\mathbb{B}}
\newcommand{\bC}{\mathbb{C}}

\newcommand{\bN}{\mathbb{N}}

\newcommand{\bR}{\mathbb{R}}

\newcommand{\cA}{\mathcal{A}}

\newcommand{\cH}{\mathcal{H}}

\newcommand{\cL}{\mathcal{L}}

\newcommand{\ra}{\rightarrow}
\newcommand{\sub}{\subseteq}

\begin{document}
\title{$p$-Summable commutators in Bergman spaces on egg domains}

\author{Mohammad Jabbari}
\address{Mohammad Jabbari, Centro de Investigacion en Matematicas, A.P. 402, Guanajuato, Gto., C.P. 36000, Mexico}
\email{mohammad.jabbari@cimat.mx}

\maketitle

%\begin{abstract}
%The exact range of the positive real parameter $p$ is determined so that the commutators in the C*-algebra of Toeplitz operators associated to continuous symbols and acting on Bergman spaces over generalized complex ellipsoids are Schatten $p$-summable.
%\end{abstract}
\begin{abstract}
The exact range of the positive real parameter $p$ is determined so that the multiplications by polynomial functions acting on Bergman spaces over complex ellipsoids are Schatten $p$-summable.
\end{abstract}

%\tableofcontents

\section{Indtroduction}
Recall that an operator $T$ on a Hilbert space is (Schatten) $p$-summable, $0<p<\infty$, if $(TT^*)^{p/2}$ is trace class \cite[Chapter 11]{dunford-s}, \cite{gohberg-krein,simon}. 
(Equivalently, if the sequence of singular values $s_j=\inf\{\|T-F\|:\mathrm{rank}(F)<j\}$ belongs to $l^p$.)
A commuting tuple $(T_1,\ldots,T_m)$ of operators on a Hilbert space $\cH$ is called $p$-essentially normal if all of the commutators $[T_j,T^{\ast}_k]$, $j,k=1,\ldots,m$ are $p$-summable.
Alternatively, $p$-essential normality can be attributed to the Hilbert $\bC[z_1,\ldots,z_m]$-module generated by $(T_1,\ldots,T_m)$, namely, $\cH$ with the module action $P(z_1,\ldots,z_m)\cdot f$, $P\in \bC[z_1,\ldots,z_m]$, $f\in\cH$ given by $P(T_1,\ldots,T_m)f$.
In noncommutative geometry, $p$-essential normality is important in the study of differentiable structures \cite{douglas-smoothness,douglas-index,douglas-voiculescu-smoothness,gong-smoothness} as well as the construction of the Chern-Connes character for noncommutative spaces \cite{connes-NCDG,connes-NCG}.

This article is about the $p$-essential normality of the multiplications by the coordinate functions acting on the Bergman space $L^2_a$ of square-integrable holomorphic functions over the egg domains of the form
\begin{equation}
\Om_1:=\left\{\sum\limits_{j=1}^{m}\left|z_j\right|^{2p_j}<1\right\}\sub\bC^m,\quad p_j>0,\label{egg1}
\end{equation}
or more generally of the form
\begin{equation}
\Om_2:=\left\{\sum_{k=1}^K\left(\sum\limits_{j=1}^{j_k}\left|z_{jk}\right|^{2p_{jk}}\right)^{a_k}<1\right\}\sub\bC^{j_1+\cdots+j_K},\quad p_{jk},a_k>0.
\label{egg2}
\end{equation}
%where $p_{jk},a_k$ are positive real numbers and $j_k$ are positive integers.
(When all $p_{jk}$ equal $1$, $\Om_2$ is called a generalized complex ellipsoid in \cite{jarnicki-pflug,kkm}.)
Note that $\Om_2$ is always pseudoconvex (because it is a logarithmically convex, complete Reinhard domain \cite[Theorem 3.28]{range}), but, assuming all $a_k\geq 1$, it is strongly pseudoconvex with $C^2$ boundary exactly when all $p_{jk}, a_k$ equal $1$. 
(Compare \cite{dangelo0}.)

To put the main result of this paper in a proper context, a brief survey of the literature related to the $p$-essential normality of the usual Hilbert modules of analytic functions is given.
In what follows, module actions are always given by multiplication with polynomial functions.
First, note that the $p$-essential normality of, say, a Bergman module is closely related to the Schatten membership of the Hankel operators $H_{\overline{f}}:=(I-P)M_{\overline{f}}P$ as well as the commutators $C_f:=[M_f,P]$ associated to coordinate functions $f$.
(Here, $M_f$ is the multiplication by $f$, and $P$ is the orthogonal projection in $L^2$ onto $L^2_a$, the so-called Bergman projection.)
% be the (big) Hankel operator associated to a function $f\in L^2$.
More precisely, the formal identities
\[
\left[M_{z_j},M_{z_k}^*\right]=C_{z_j}M_{\overline{z_k}},
\]
\[
C_f=H_f-H_{\overline{f}}^*,\quad
(I-P)C_f=H_f,\quad
C_f(I-P)=-H_{\overline{f}}^*
\]
show that the Bergman module is $p$-essentially normal if all commutators $C_f$ associated to coordinate functions are $p$-summable, and that $C_f$ is $p$-summable if and only if both $H_f$ and $H_{\overline{f}}$ are so.
(Obviously, $H_f$ is zero for holomorphic functions $f$.)
Here is our survey:
\begin{itemize}
\item
The Bergman module on the unit ball of $\bC^m$ is $p$-essentially normal exactly when $p>1/2$ and $m=1$, or $p>m>1$ \cite{arazy-fisher-peetre,arazy-fisher-janson-peetre}.
The $m$-shift (or Drury-Arveson) module $H^2_m$ is $p$-essentially normal if and only if $p>0$ and $m=1$, or $p>m>1$ \cite{arveson-dilation}.

\item
The Hardy module on a smoothly bounded, strongly pseudoconvex domain is $p$-essentially normal if $p$ is strictly larger than the complex dimension of the domain \cite[Theorem 13.1]{boutetdemonvel-guillemin}, \cite[Proposition 1]{ee}.

\item
The Schatten membership of the commutators $C_f$, as well as the Hankel operators, are studied in:
\begin{itemize}
\item 
\cite{arazy-fisher-peetre,arazy-fisher-peetre-planar,arazy-fisher-janson-peetre,isralowitz-bergman,janson,pau,raimondo,rochberg-semmes,wallsten,xia-1,xia-2,zhu-schatten,zhu-OT}
for the Bergman space on the unit ball.

\item
\cite{zheng} 
for the Bergman space on bounded symmetric domains.

\item
\cite{fang-xia-schatten,feldman-rochberg,peller,peller-vectorvalued,peller-lessthanone,peller-book,rochberg,semmes,zhu-FT}
for the Hardy space on the unit sphere.

\item
\cite{li.H,li.H-luecking,peloso}
for the Hardy space on smoothly bounded, strongly pseudoconvex domains.

\item
\cite{isralowitz-bargmann,xia-zheng} 
for the Segal-Bargmann (or Fock) space.

\item
\cite{beatrous-li-2} for $P$ being a Calderon-Zygmund operator associated to a homogeneous space.
%\cite{beatrous-li-2} studies the Schatten membership of the commutator of a multiplication operator with a Calderon-Zygmund  operator in the setting of homogeneous spaces.
\end{itemize}

\item
The essential normality of the Bergman module on a  bounded, pseudoconvex domain  is equivalent to the compactness of the $\overline{\pa}$-Neumann operator $N_1$ on $(0,1)$-forms with $L^2$ coefficients \cite{catlin-dangelo,fu-straube,salinas,salinas-sheu-upmeier}.
(The essential normality means that the C*-algebra generated by multiplications with coordinate functions is commutative modulo the ideal of compact operators.)
Several sufficient conditions for this are given in the literature \cite{catlin-globalregularity,catlin-boundaryinvariants,henkin-iordan,mcneal-compactness}. 
For example, strongly pseudoconvex domains, domains of finite type, and pseudoconvex domains with real analytic boundary have compact $N_{1}$.
Other approaches to the essential normality of Bergman modules are given in \cite{axler,beatrous-li-1,bekolle-berger-coburn-zhu,curto-muhly,krantz-li,salinas-sheu-upmeier,upmeier}.

%\item
%\cite{beatrous-li-2} studies 

%There are many works in the literature which study the compactness of commutators $C_f$ as well as Hankel operators:
%\cite{axler} (for Bergman spaces on the unit disk), 
%\cite{bekolle-berger-coburn-zhu} (for Bergman space on symmetric bounded domains),
%\cite{coifman-rochberg-weiss} (for Hardy spaces),
%\cite{salinas,salinas-sheu-upmeier}.
%Specially, \cite{beatrous-li-2} studies the Schatten membership of the commutator of a multiplication operator with a Calderon-Zygmund  operator in the setting of homogeneous spaces.
\end{itemize}

The first main result  of this article is: 

\begin{theorem}\label{theorem1}
The Bergman module on the domain (\ref{egg1}) is $p$-essentially normal if and only if 
\[
p>\begin{cases}
\frac{1}{2},&m=1,\\
\max\left\{
m,
p_j(m-1): j=1,\ldots,m\right\},&m>1.
\end{cases}
\]
\end{theorem}

This result was first stated without proof in \cite{jabbari-egg-index}.
Our next result generalizes this:

\begin{theorem}\label{theorem2}
The Bergman module on the domain $\Om_2$ given in (\ref{egg2}) is $p$-essentially normal if and only if 
\[
p>\begin{cases}
\frac{1}{2},&d=1,\\
\max\left\{
d,
q_1,\ldots,q_K
\right\},&d>1,
\end{cases}
\]
where
\[
d=j_1+\cdots+j_K
\]
is the complex dimension of $\Om_2$, and for each $k=1,\ldots,K$,
\[
q_k
=\begin{cases}
\max\left\{a_kp_{jk}(d-j_k):j=1,\ldots,j_k\right\},& j_k=1,\\
\max\left\{p_{jk}(d-1),
a_kp_{jk}(d-j_k):j=1,\ldots,j_k\right\},& j_k>1=a_k,\\
\max\left\{p_{jk}(d-1),
a_kp_{jk}(d-j_k),
\frac{2a_k(d-j_k)}{1/p_{jk}+1/p_{lk}}:j,l=1,\ldots,j_k,\ j\neq l\right\},& j_k>1\neq a_k.
\end{cases}
\]
\end{theorem}

%The author believes that the converse of Theorem \ref{theorem2} is also true, but he has not yet been able to provide an argument.
%(This is accomplished by proving the converse of Lemma \ref{fact-2} below.)

An interesting phenomenon appearing here (already observed in \cite{krantz-li-rochberg}; see also \cite{jabbari-egg-index}) is that for weakly pseudoconvex domains, in contrast to strongly pseudoconvex ones, the cut-off values for the $p$-summability of the Bergman modules depend not only on the dimensions of the domains but also on their boundary geometry. 
In our case of the egg domains of type (\ref{egg1}) or (\ref{egg2}), this geometry is manifested in the maximum order of contact of the boundary with complex analytic curves (in the sense of D'Angelo \cite{dangelo-annals,dangelo-book}).
%In another direction, the forthcoming article \cite{jabbari-schatten} studies the Chern-Connes character of the Toeplitz algebra associated to the Bergman space on egg domains.

This paper is organized as follows:
I gather several facts needed in the proof of the main results in Section \ref{section-preliminaries}.
Theorems \ref{theorem1} and \ref{theorem2} are respectively proved in Sections \ref{section-theorem1} and \ref{section-theorem2}.

\section{Preliminaries}\label{section-preliminaries}
This section provides some facts needed in the proof of Theorems \ref{theorem1} and \ref{theorem2}.

\begin{lemma}\label{fact-gamma}
Given positive real numbers $a,b$, and real variable $x$, we have
\[
\frac{\Gamma(x+a)}{\Gamma(x+b)}x^{b-a}
=1+\frac{(a-b)(a+b-1)}{2x}+\frac{(a-b)(a-b-1)\left(3(a+b-1)^2-a+b-1\right)}{24x^2}+O\left(x^{-3}\right),
\]
\[
\frac{\Gamma(x+a)^2}{\Gamma(x)\Gamma(x+2a)}
=
1-\frac{a^2}{x}+\frac{a^2(a^2+2a-1)}{2x^2}+O\left(x^{-3}\right),
\]
\[
\frac{\Gamma(x+a)\Gamma(x+2a+b)}{\Gamma(x+a+b)\Gamma(x+2a)}
=
1+\frac{ab}{x}+\frac{ab(ab-3a-b+1)}{x^2}+O\left(x^{-3}\right),
\]
\[
\frac{(x+a)^2}{x(x+2a)}
= 
1+\frac{a^2}{x^2}+O\left(x^{-3}\right),
\]
\[
\frac{(x+a)(x+2a+b)}{(x+a+b)(x+2a)}
=
1-\frac{ab}{x^2}+O\left(x^{-3}\right),
\]
as $x\ra\infty$.
\end{lemma}

\begin{proof}
The first formula is proved in \cite{et} as well as \cite[Appendix C]{andrews}.
The next two are immediate from it.
\end{proof}

\begin{lemma}\label{fact-sumoftwoseries}
Let $\sum a_j$ and $\sum b_j$ be two series with nonnegative terms, and let $p>0$.
Then $\sum (a_j+b_j)^p$ converges if and only if both $\sum a_j^p$ and $\sum b_j^p$ converge.
\end{lemma}

\begin{proof}
Immediate from 
\[
\max\left\{a_j^p,b_j^p\right\}\leq \left(a_j+b_j\right)^p\leq \left(2\max\left\{a_j,b_j\right\}\right)^p
	\leq 2^p\left(a_j^p+b_j^p\right).
	\]
\end{proof}

The proof of Theorem \ref{theorem1} depends on the following fact about the convergence of higher zeta series.
(Compare \cite{matsumoto,tornheim}.)
In what follows, $\bN$ and $\bN_+$ respectively denote the set of nonnegative and positive integers.

\begin{lemma}\label{fact-1}
	Suppose positive integer $m$ and real numbers $b, a_1, \ldots,a_m$.
	Then:
	
	(a) The series
	\[
	A:=\sum_{i\in\bN_+^m}\frac{i_1^{a_1}\cdots i_m^{a_m}}{(i_1+\cdots+i_m)^b}
	\]
	converges if and only if
	\begin{equation}
	b>\max\left\{\sum_{j\in J} (a_j+1):\emptyset\neq J\sub\{1,\ldots,m\}\right\}.\label{iff}
	\end{equation}
	Specially, if all $a_j$ are $\geq -1$, then $A$ converges if and only if 
	\[
	b>a_1+\cdots+a_m+m.
	\]
	
	(b)
	Given positive integer $k$ and real number $a$, the series
	\[
	B:=\sum_{i\in\bN_+^{m+k}}\frac{i_1^{a_1}\cdots i_{m}^{a_m}(i_{m+1}+\cdots+i_{m+k})^a}{(i_1+\cdots+i_{m+k})^b}
	\]
	converges if and only if the series 
	\[
	\sum_{i\in\bN_+^{m+1}}\frac{i_1^{a_1}\cdots i_{m}^{a_m}i_{m+1}^{a+k-1}}{(i_1+\cdots+i_{m+1})^b}
	\]
converges, if and only if
	\begin{equation*}
	b>\max\left\{\sum_{j\in J} (a_j+1),a+k,a+k+\sum_{j\in J} (a_j+1):\emptyset\neq J\sub\{1,\ldots,m\}\right\}.
	\end{equation*}	
\end{lemma}
\begin{proof}
	(a)
	We prove by induction on $m$ that $A<\infty$ if and only if (\ref{iff}) holds, if and only if 
	\[
	A':=\sum_{i\in\bN_+^m} \frac{i_1^{a_1}\cdots i_m^{a_m}\log(i_1+\cdots+i_m)}{(i_1+\cdots+i_m)^b}<\infty.
	\]
	The case $m=1$ is clear, so assume $m>1$.
	We split the proof into two cases.
	
	\textit{Case I: All $a_j$ are nonnegative.}
	In this case, $A$ is dominated by $\sum (i_1+\cdots+i_m)^{-c}$, where $c=b-\sum a_j$.
	This latter series equals $\sum_{N\in\bN_+}N^{-c}\mu$, 
	where $\mu=\binom{N-1}{m-1}$ is the number of tuples $(i_1,\ldots,i_m)$ of positive integers that sum up to $N$.
	Since $\mu$ grows like $N^{m-1}$ when $N\ra\infty$, it follows that $A$ converges if $c>m$, namely if $b>\sum a_j+m$.
	Conversely, we show that $A$ diverges when $b=\sum a_j+m$.
	Since $b>a_m$, the general summand of $A$ is decreasing with respect to $i_{m}$ provided that $i_m\geq r(i_1+\cdots+i_{m-1})$, where $r=a_{m}/(b-a_{m})>0$.
	By the integral test, it suffices to show that the series
	\[
	A_1:=\sum_{i\in\bN_+^{m-1}} i_1^{a_1}\cdots i_{m-1}^{a_{m-1}}\int_{r(i_1+\cdots+i_{m-1})+1}^{\infty}\frac{x^{a_{m}}dx}{(i_1+\cdots+i_{m-1}+x)^b}
	\]
	diverges.
	After the change of variables $x=(i_1+\cdots+i_{m-1})t$,
	\[
	A_1
	=\sum\frac{i_1^{a_1}\cdots i_{m-1}^{a_{m-1}}}{(i_1+\cdots+i_{m-1})^{b-a_{m}-1}}\int_{r+\frac{1}{i_1+\cdots+i_{m-1}}}^{\infty}\frac{t^{a_{m}}dt}{(t+1)^b}\\
	\geq\sum\frac{i_1^{a_1}\cdots i_{m-1}^{a_{m-1}}}{(i_1+\cdots+i_{m-1})^{b-a_{m}-1}}\int_{r+1}^{\infty}\frac{t^{a_{m}}dt}{(t+1)^b}.
	\]
	By the induction hypothesis, $A_1$ diverges.
	So far, we have shown that $A<\infty$ if and only if (\ref{iff}) holds.
	If (\ref{iff}) holds, choosing $\e>0$ such that (\ref{iff}) still holds when $b$ is replaced by $b-\e$, then $A'$ converges because it is dominated by the convergent series
	\[
	\sum_{i\in\bN_+^m}\frac{i_1^{a_1}\cdots i_m^{a_m}}{(i_1+\cdots+i_m)^{b-\e}}.
	\]
	Clearly, the convergence of $A'$ implies the convergence of $A$.
	
	\textit{Case II: At least one $a_j$ is negative.}
	Without loss of generality suppose $a_m<0$.
	First assume that $A$ converges.
	Then, by the integral test, 
	\[
	A_2:=\sum_{i\in\bN_+^{m-1}} i_1^{a_1}\cdots i_{m-1}^{a_{m-1}}\int_{1}^{\infty}\frac{x^{a_{m}}dx}{(i_1+\cdots+i_{m-1}+x)^b}<\infty.
	\]
	After the change of variables $x=(i_1+\cdots+i_{m-1})t$,
	\[
	A_2
	=\sum\frac{i_1^{a_1}\cdots i_{m-1}^{a_{m-1}}}{(i_1+\cdots+i_{m-1})^{b-a_{m}-1}}\int_{\frac{1}{i_1+\cdots+i_{m-1}}}^{\infty}\frac{t^{a_{m}}dt}{(t+1)^b}.
	\]
	Since
	\[
	A_2\geq \sum\frac{i_1^{a_1}\cdots i_{m-1}^{a_{m-1}}}{(i_1+\cdots+i_{m-1})^{b-a_{m}-1}}\times\int_{1}^{\infty}\frac{t^{a_{m}}dt}{(t+1)^b},
	\]
	it follows that both 
	\[
	A_3
	:=\sum\frac{i_1^{a_1}\cdots i_{m-1}^{a_{m-1}}}{(i_1+\cdots+i_{m-1})^{b-a_{m}-1}}\quad\mathrm{and}\quad I:=\int_{1}^{\infty}\frac{t^{a_{m}}dt}{(t+1)^b}
	\]
	converge.
	The integral $I$ converges exactly when
	\begin{equation}
	b>a_{m}+1.\label{iff0}
	\end{equation}
	By the induction hypothesis and Case I, $A_3$ converges exactly when
	\begin{equation}
	b>\max\left\{a_{m}+1+\sum_{j\in J} (a_j+1):J\sub\{1,\ldots,m-1\}\right\}.\label{iff1}
	\end{equation}
	Therefore, the convergence of $A_2$ implies the convergence of 
	\[
	A_4:=\sum\frac{i_1^{a_1}\cdots i_{m-1}^{a_{m-1}}}{(i_1+\cdots+i_{m-1})^{b-a_{m}-1}}\int_{\frac{1}{i_1+\cdots+i_{m-1}}}^{1}\frac{t^{a_{m}}dt}{(t+1)^b}.
	\]
	Since $(t+1)^{-b}$ is bounded between $2^{-b}$ and $1$ when $(i_1+\cdots+i_{m-1})^{-1}\leq t\leq 1$, it follows that the convergence of $A_4$ is equivalent to the convergence of 
	\[
	A_5
	:=\sum\frac{i_1^{a_1}\cdots i_{m-1}^{a_{m-1}}}{(i_1+\cdots+i_{m-1})^{b-a_{m}-1}}\int_{\frac{1}{i_1+\cdots+i_{m-1}}}^{1}t^{a_{m}}dt.
	\]
	According to whether $a_{m}<-1$, $a_{m}=-1$, or $a_{m}>-1$, the convergence of $A_5$ is respectively equivalent to the convergence of the series:
	\[
	\sum\frac{i_1^{a_1}\cdots i_{m-1}^{a_{m-1}}}{(i_1+\cdots+i_{m-1})^{b}},
	\]
	\begin{equation*}
	\sum\frac{i_1^{a_1}\cdots i_{m-1}^{a_{m-1}}\log(i_1+\cdots+i_{m-1})}{(i_1+\cdots+i_{m-1})^{b}},\label{logg}
	\end{equation*}
	or
	\[
	\sum\frac{i_1^{a_1}\cdots i_{m-1}^{a_{m-1}}}{(i_1+\cdots+i_{m-1})^{b-a_{m}-1}}.
	\]
	By the induction hypothesis and Case I, it follows that the convergence of $A_2$ implies ((\ref{iff0}) and (\ref{iff1})) together with 
	\begin{equation}
	b>\max\left\{\sum_{j\in J} (a_j+1):J\sub\{1,\ldots,m-1\}\right\}\ \mathrm{whenever}\ a_{m}<-1.\label{iff2}
	\end{equation}
	Clearly, ((\ref{iff0}) $\mathrm{and}$ (\ref{iff1}) $\mathrm{and}$ (\ref{iff2})) is logically equivalent to (\ref{iff}).
	Conversely, assume that (\ref{iff}) holds.
	Reversing the arguments above shows that $A_2$ converges.
	Then, by the integral test,
	\[
	A_6:=\sum_{i\in\bN_+^m,i_m>1}\frac{i_1^{a_1}\cdots i_m^{a_m}}{(i_1+\cdots+i_m)^b}<\infty.
	\]
	Note that 
	\[
	A
	=
	A_6+
	\sum_{i\in\bN_+^{m-1}}\frac{i_1^{a_1}\cdots i_{m-1}^{a_{m-1}}}{(i_1+\cdots+i_{m-1}+1)^b}
	\leq A_6+\sum_{i\in\bN_+^{m-1}}\frac{i_1^{a_1}\cdots i_{m-1}^{a_{m-1}}}{C^b(i_1+\cdots+i_{m-1})^b},
	\]
	where $C$ equals $1$ if $b>0$, and equals $2$ otherwise.
	Since the second series converges by the induction hypothesis and Case I, it follows that $A<\infty$.
	We have shown that $A<\infty$ if and only if (\ref{iff}) holds.
	Similar to Case I, one can show that (\ref{iff}) is equivalent to $A'<\infty$.

	(b)
Note that  
	\[
	B=\sum_{i\in\bN_+^{m+1},\ N\in\bN_+}\frac{i_1^{a_1}\cdots i_m^{a_m}N^a\mu}{(i_1+\cdots+i_m+N)^b},
	\]
	where $\mu=\binom{N-1}{k-1}$ is the number of tuples $(i_{m+1},\ldots,i_{m+k})$ of positive integers that sum up to $N$.
	To study the convergence, one can replace $\mu$ by $N^{k-1}$, and the result follows from (a).
\end{proof}

During the proof of Theorem \ref{theorem2}, we will need to know exactly when each of the following higher zeta series (beside the ones appearing in Lemma \ref{fact-1}) converge:

%In order to determine the exact range of $p$ such that the Bergman module over the domain given in (\ref{egg2}) is $p$-essential normal, we need to know exactly when each of the higher zeta series of the following forms converge:
\begin{equation*}
\sum\frac{i_1^{a_1}i_2^{a_2}i_3^{a_3}(i_1+i_2)^{a_{12}}}{(i_1+i_2+i_3)^b},
\end{equation*}
\begin{equation*}
\sum\frac{i_1^{a_1}i_2^{a_2}i_3^{a_3}i_4^{a_4}(i_1+i_2+i_3)^{a_{123}}}{(i_1+i_2+i_3+i_4)^b},
\end{equation*}
\begin{equation*}
\sum\frac{i_1^{a_1}i_2^{a_2}i_3^{a_3}i_4^{a_4}(i_1+i_2+i_3)^{a_{123}}|i_1+i_2+i_3-i_4|^{a_{1234}}}{(i_1+i_2+i_3+i_4)^b},
\end{equation*}
\begin{equation*}
\sum\frac{i_1^{a_1}i_2^{a_2}i_3^{a_3}i_4^{a_4}i_5^{a_5}(i_1+i_2)^{a_{12}}(i_3+i_4)^{a_{34}}}{(i_1+i_2+i_3+i_4+i_5)^b}.
\end{equation*}
Here, all parameters $b,a_1,a_2,\ldots$ are real numbers.
The following two lemmas solve this problem.
Their proof is modeled on the arguments in \cite[Lemma 2]{tornheim}.

\begin{lemma}\label{fact-2}
Let $\cL$ be a finite collection of linear functionals $l=l(i_1,\ldots,i_m):\bN_+^m\ra\bR$ with nonnegative coefficients $\pa l/\pa i_j$, $j=1,\ldots,m$.
Suppose that to each $l\in\cL$, there associates a real number $a$.
Let $b$ be a real number.
If the infinite series 
\[
S:=\sum_{i\in\bN_+^m}\frac{\prod_{l\in\cL} l(i_1,\ldots,i_m)^a}{(i_1+\cdots+i_m)^b}
\] 
converges, then
\begin{equation}
b>\max\left\{|J|+\sum_{J\sim l,\ l\in\cL} a \ :\ \emptyset\neq J\sub\{1,\ldots,m\}\right\},\label{iff-b}
\end{equation}
where $|J|$ is the cardinality of $J$, and the notation $J\sim l$ indicates that $\pa l/\pa i_j\neq 0$ for some $j\in J$.
The same is true if one $l(i_1,\ldots,i_m)$ is replaced by $|-i_1+i_2+\cdots+i_m|$.
\end{lemma}

\begin{proof}
We argue by induction on $m$.
Since the case $m=1$ is clear, assume $m>1$. 
Note that (\ref{iff-b}) can be written as
\begin{equation}
b>\max\{b_1,b_1+b_{12},b_2\},\label{iff-b'}
\end{equation}
where
\[
b_1:=1+\sum_{\{1\}\sim l}a,
\]
\[
b_{12}:=\max\left\{|J|+\sum_{J\sim l,\ \{1\}\not\sim l}a\ :\   \emptyset\neq J\sub\{2,\ldots,m\} \right\},
\]
\[
b_{2}:=\max\left\{|J|+\sum_{J\sim l} a\ :\ \emptyset\neq J\sub\{2,\ldots,m\}\right\}.
\]

Define the function
\begin{equation}
\chi:\bR\ra\bR,\quad 
\chi(x)
=\begin{cases}
1,&\text{if}\ x\geq 0,\\
2,&\text{if}\ x<0.\label{charfunction}
\end{cases}
\end{equation}

The convergence of $S$ implies the convergence of 
\begin{equation}
\sum_{i_1\geq 1=i_2=\cdots=i_m}\frac{\prod l(i_1,\ldots,i_m)^a}{(i_1+\cdots+i_m)^b}\label{first}
\end{equation}
as well as 
\begin{equation}
\sum_{i_1\geq i_2+\cdots+i_m}\frac{\prod l(i_1,\ldots,i_m)^a}{(i_1+\cdots+i_m)^b}\label{second}
\end{equation}
The convergence of the series (\ref{first}) implies $b>b_1$.
Note that on the region $i_1\geq i_2+\cdots+i_m$, two quantities $i_1$ and $i_1+\cdots+i_m$ are comparable in the sense that 
\begin{equation}
\frac{i_1+\cdots+i_m}{2}\leq i_1\leq i_1+\cdots+i_m.\label{estimate}
\end{equation}
Since $b>b_1$, this shows that the series (\ref{second}) dominates
\begin{multline*}
\sum_{i_1\geq i_2+\cdots+i_m}\frac{\prod l\left(\frac{i_1+\cdots+i_m}{\chi(-a)},i_2,\ldots,i_m\right)^a}{(i_1+\cdots+i_m)^b}\gtrsim
\sum_{i_1\geq i_2+\cdots+i_m}\frac{\prod_{\{1\}\not\sim l} l^a}{(i_1+\cdots+i_m)^{b-b_1+1}}
\\
\geq
\sum_{i_2,\ldots,i_m\geq 1}\int_{i_2+\cdots+i_m}^{\infty}\frac{\prod_{\{1\}\not\sim l} l^a}{(i_1+\cdots+i_m)^{b-b_1+1}}di_1
\cong
\sum \frac{\prod_{\{1\}\not\sim l} l^a}{(i_2+\cdots+i_m)^{b-b_1}}.
\end{multline*}
Therefore, by the induction hypothesis, $b-b_1>b_{12}$.
We have shown that part of (\ref{iff-b}) which contains the contribution from $1\in J$.
Similar arguments (namely, summing over each of the regions $i_j\geq\sum_{k\neq j} i_k$, $j=2,\ldots,m$ instead of $i_1\geq\sum_{k=2}^m i_k$) imply $b>b_2$.

Finally, assume that one $l(i_1,\ldots,i_m)$ is of the form $|-i_1+i_2+\cdots+i_m|$.
The proof above works after making the following modifications: 
\begin{enumerate}
\item 
In the proof of $b>b_1+b_{12}$, replace the region $i_1\geq i_2+\cdots+i_m$ with $i_1\geq 3(i_2+\cdots+i_m)$, and the comparability estimate (\ref{estimate}) with
\[
\frac{i_1+3(i_2+\cdots+i_m)}{2}\leq i_1\leq i_1+2(i_2+\cdots+i_m).
\]

\item
In the proof of $b>b_2$, for each $j=2,\ldots,m$, replace the region $i_j\geq\sum_{k\neq j}i_k$ with $i_j\geq 3\sum_{k\neq j}i_k$, and the comparability estimate (\ref{estimate}) with
\[
\frac{i_j+3\sum_{k\neq j}i_k}{2}\leq i_j\leq i_j+2\sum_{k\neq j}i_k.
\]
\end{enumerate}
The proof is complete.
\end{proof}

\begin{lemma}\label{fact-3}
The converse of Lemma \ref{fact-2} holds if $S$ is any of the following series:

(a)
\begin{equation*}
A:=
	\sum_{i\in\bN^3_+}\frac{i_1^{a_1}i_2^{a_2}i_3^{a_3}(i_1+i_2)^{a_{12}}}{(i_1+i_2+i_3)^b}.\label{A}
	\end{equation*}

(b)
	\begin{equation*}
	B:=
	\sum_{i\in\bN^4_+}\frac{i_1^{a_1}i_2^{a_2}i_3^{a_3}i_4^{a_4}(i_1+i_2)^{a_{12}}}{(i_1+i_2+i_3+i_4)^b}.\label{B}
	\end{equation*}
	
(c)
	\begin{equation*}
	C:=
	\sum_{i\in\bN^4_+}\frac{i_1^{a_1}i_2^{a_2}i_3^{a_3}i_4^{a_4}(i_1+i_2+i_3)^{a_{123}}}{(i_1+i_2+i_3+i_4)^b}.\label{C}
	\end{equation*}
	
(d)
	\begin{equation*}
	D:=
	\sum_{i\in\bN^4_+}\frac{i_1^{a_1}i_2^{a_2}i_3^{a_3}i_4^{a_4}(i_1+i_2+i_3)^{a_{123}}|i_1+i_2+i_3-i_4|^{a_{1234}}}{(i_1+i_2+i_3+i_4)^b},\quad a_{1234}>0.
	\label{D}
	\end{equation*}
	
(e)
	\begin{equation*}
	E:=
	\sum_{i\in\bN^4_+}\frac{i_1^{a_1}i_2^{a_2}i_3^{a_3}i_4^{a_4}(i_1+i_2)^{a_{12}}(i_3+i_4)^{a_{34}}}{(i_1+i_2+i_3+i_4)^b},\quad a_1+a_{12}>-1.\label{E}
	\end{equation*}

(f)
\begin{equation*}
	F:=
	\sum_{i\in\bN^5_+}\frac{i_1^{a_1}i_2^{a_2}i_3^{a_3}i_4^{a_4}i_5^{a_5}(i_1+i_2)^{a_{12}}(i_3+i_4)^{a_{34}}}{(i_1+i_2+i_3+i_4+i_5)^b},\quad a_1+a_{12}>-1.\label{F}
	\end{equation*}
\end{lemma}
\begin{proof}
(a)
We should show that $A$ converges assuming
\begin{equation}
b>\max\{b_1,b_2,b_1+b_{2}\},\label{m3}
\end{equation}
where
\[
b_1:=a_3+1,
\]
\[
b_2:=\max\{a_1+a_{12}+1,a_2+a_{12}+1,a_1+a_2+a_{12}+2\}.
\]
Since
\[
A\leq A_1+A'_2+A_2'',
\]
\[
A_1:=\sum_{i_3\geq i_1+i_2}\frac{i_1^{a_1}i_2^{a_2}i_3^{a_3}(i_1+i_2)^{a_{12}}}{(i_1+i_2+i_3)^b},
\]
\[
A'_2:=\sum_{1=i_3\leq i_1+i_2}\frac{i_1^{a_1}i_2^{a_2}i_3^{a_3}(i_1+i_2)^{a_{12}}}{(i_1+i_2+i_3)^b},
\]
\[
A_2'':=\sum_{1<i_3\leq i_1+i_2}\frac{i_1^{a_1}i_2^{a_2}i_3^{a_3}(i_1+i_2)^{a_{12}}}{(i_1+i_2+i_3)^b},
\]
it suffices to show that $A_1$, $A_2'$ and $A''_2$ converge.
Since $b>b_1$,
\begin{multline*}
A_1
\leq
\sum_{i_3\geq i_1+i_2}\frac{i_1^{a_1}i_2^{a_2}(i_1+i_2)^{a_{12}}\left(\frac{i_1+i_2+i_3}{\chi(a_{3})}\right)^{a_{3}}}{(i_1+i_2+i_3)^b}
\leq\sum_{i_1,i_2\geq 1}\int_{i_1+i_2-1}^{\infty}\frac{i_1^{a_1}i_2^{a_2}(i_1+i_2)^{a_{12}}}{(i_1+i_2+i_3)^{b-b_1+1}}di_3\\
\lesssim\sum \frac{i_1^{a_1}i_2^{a_2}}{(i_1+i_2)^{b-b_1-a_{12}}}.
\end{multline*}
Therefore, $A_1<\infty$ according to Lemma \ref{fact-1}.(a) and because $b>b_1+b_2$.
We have $A'_2<\infty$ because $b>b_2$.
%It remains to show that $C''_2<\infty$.
The series $A''_2$ is dominated by 
\begin{equation*}
\sum_{1<i_3\leq i_1+i_2}\frac{i_1^{a_1}i_2^{a_2}(i_1+i_2)^{a_{12}}i_3^{a_3}}{\left(\chi(b)(i_1+i_2)\right)^b}
\lesssim
\sum_{i_1,i_2\geq 1}\int_{1}^{i_1+i_2+1}\frac{i_1^{a_1}i_2^{a_2}i_3^{a_3}}{(i_1+i_2)^{b-a_{12}}}di_3.
\end{equation*}
Therefore, if $a_{3}>-1$, then $A''_2$ is dominated by 
\[
\sum\frac{i_1^{a_1}i_2^{a_2}}{(i_1+i_2)^{b-a_{12}-b_1}},
\]
and this latter series converges according to Lemma \ref{fact-1}.(a) and because $b>b_1+b_2$.
If $a_3<-1$, then $A_2''$ is dominated by 
\[
\sum\frac{i_1^{a_1}i_2^{a_2}}{(i_1+i_2)^{b-a_{12}}},
\]
and this latter series converges according to Lemma \ref{fact-1}.(a) and because $b>b_2$.
If $a_3=-1$, then $A''_2$ is dominated by
\[
\sum\frac{i_1^{a_1}i_2^{a_2}\log(i_1+i_2)}{(i_1+i_2)^{b-a_{12}}}
\lesssim
\sum\frac{i_1^{a_1}i_2^{a_2}}{(i_1+i_2)^{b-a_{12}-\e}},
\]
where we have chosen $\e>0$ small enough such that (\ref{m3}) holds for $b$ replaced by $b-\e$.
This latter series is finite 
according to Lemma \ref{fact-1}.(a) and because $b-\e>b_2$.

(b, c)
The series $B$ and $C$ are treated similar to $A$, now summing over regions $i_4\geq i_1+i_2+i_3$, $1=i_4\leq i_1+i_2+i_3$, and $1<i_4\leq i_1+i_2+i_3$.

(d)
Since $a_{1234}>0$, the series $D$ is dominated by
\begin{equation*}
\sum\frac{i_1^{a_1}i_2^{a_2}i_3^{a_3}i_4^{a_4}(i_1+i_2+i_3)^{a_{123}}}{(i_1+i_2+i_3+i_4)^{b-a_{1234}}},
\end{equation*}
and we are reduced to (c).

(e)
Assume that the condition (\ref{iff-b}) holds for the series $E$.
If $a_{12}\geq 0$, then, by Lemma \ref{fact-sumoftwoseries}, it suffices to show that both of the following series converge:
\[
\sum\frac{i_1^{a_1+a_{12}}i_2^{a_2}i_3^{a_3}i_4^{a_4}(i_3+i_4)^{a_{34}}}{(i_1+i_2+i_3+i_4)^b},
\]
\[
\sum\frac{i_1^{a_1}i_2^{a_2+a_{12}}i_3^{a_3}i_4^{a_4}(i_3+i_4)^{a_{34}}}{(i_1+i_2+i_3+i_4)^b}.
\]
Both of these series converge by (b).
Now, assume $a_{12}< 0$.
Since
\[
E\leq E_1+E'_2+E_2'',
\]
\[
E_1
:=\sum_{i_1\geq i_2+i_3+i_4}\frac{i_1^{a_1}i_2^{a_2}i_3^{a_3}i_4^{a_4}(i_1+i_2)^{a_{12}}(i_3+i_4)^{a_{34}}}{(i_1+i_2+i_3+i_4)^b},
\]
\[
E'_2
:=\sum_{1=i_1\leq i_2+i_3+i_4}\frac{i_1^{a_1}i_2^{a_2}i_3^{a_3}i_4^{a_4}(i_1+i_2)^{a_{12}}(i_3+i_4)^{a_{34}}}{(i_1+i_2+i_3+i_4)^b},
\]\[
E_2'':=\sum_{1<i_1\leq i_2+i_3+i_4}\frac{i_1^{a_1}i_2^{a_2}i_3^{a_3}i_4^{a_4}(i_1+i_2)^{a_{12}}(i_3+i_4)^{a_{34}}}{(i_1+i_2+i_3+i_4)^b},
\]
it suffices to show that $E_1$, $E_2'$ and $E''_2$ converge.
That $E_1$ and $E'_2$ converge can be proved by similar arguments given in (a) for the convergence of $A_1$ and $A'_2$, respectively.
Since $a_{12}<0$ and $a_{1}+a_{12}>-1$, $E''_2$ is dominated by 
\begin{multline*}
\sum_{1<i_1\leq i_2+i_3+i_4}\frac{i_1^{a_1+a_{12}}i_2^{a_2}i_3^{a_3}i_4^{a_4}(i_3+i_4)^{a_{34}}}{(i_1+i_2+i_3+i_4)^{b}}
\lesssim
\sum_{1<i_1\leq i_2+i_3+i_4}\frac{i_1^{a_1+a_{12}}i_2^{a_2}i_3^{a_3}i_4^{a_4}(i_3+i_4)^{a_{34}}}{(\chi(b)(i_2+i_3+i_4))^{b}}\\
\lesssim
\sum_{1<i_1\leq i_2+i_3+i_4}\int_{1}^{i_2+i_3+i_4+1}\frac{i_1^{a_1+a_{12}}i_2^{a_2}i_3^{a_3}i_4^{a_4}(i_3+i_4)^{a_{34}}}{(i_2+i_3+i_4)^{b}}di_1
\lesssim
\sum_{i_2,i_3,i_4\geq 1}\frac{i_2^{a_2}i_3^{a_3}i_4^{a_4}(i_3+i_4)^{a_{34}}}{(i_2+i_3+i_4)^{b-a_1-a_{12}-1}}.
\end{multline*}
The last series converges by (b).

(f)
Arguing as in (a), namely summing over regions $i_5\geq i_1+i_2+i_3+i_4$, $1=i_5\leq i_1+i_2+i_3+i_4$, and $1<i_5\leq i_1+i_2+i_3+i_4$, reduces us to (e).
\end{proof}

\begin{remark}
The author guesses that the converse of Lemma \ref{fact-2} is true in general, but he has not yet been able to prove this.
Lemma \ref{fact-3} is enough to prove Theorem \ref{theorem2}.
\end{remark}

\section{Proof of Theorem \ref{theorem1}}\label{section-theorem1}
This section proves Theorem \ref{theorem1} about the $p$-essential normality of the Bergman module $L^2_a(\Om_1)$ over the domain $\Om_1$ given in (\ref{egg1}).
Since $\Om_1$ is a complete Reinhardt domain, polynomials are dense in $L^2_a(\Om_1)$ with respect to the topology of uniform convergence on compact subsets \cite[Page 47]{range}.
Then a standard shrinking argument (\cite[Page 43]{zhu-FT}, \cite[Page 11]{ds}) shows that the normalized monomials
\[
b_i:=\frac{z^i}{\sqrt{\om_1(i)}},\quad
i\in\bN^m,
\]
where
\[
\om_1(i):=\left\|z^i\right\|^2_{L^2_a(\Om_1)},
\]
constitute an orthonormal basis for the Hilbert space $L^2_a(\Om_1)$.
An explicit formula for the norm of monomials is given by:

\begin{proposition}\label{lemma-norm-1}
Given multi-index $i\in\bN^m$, we have
	\[
	\om_1(i)
	=\frac{\pi^m}{\prod p_j}\frac{B\left(\frac{i+1}{p}\right)}{\left|\frac{i+1}{p}\right|},
	\]
	where $\frac{i+1}{p}=\left(\frac{i_1+1}{p_1},\ldots,\frac{i_m+1}{p_m}\right)$, 
	$\left|\frac{i+1}{p}\right|=\sum_{j=1}^{m}\frac{i_j+1}{p_j}$, and 
	\[
	B\left(\frac{i+1}{p}\right)
	=\frac{\prod_{j=1}^m\Gamma\left(\frac{i_j+1}{p_j}\right)}{\Gamma\left(\sum_{j=1}^m\frac{i_j+1}{p_j}\right)}
	\]
	is the multi-variable Beta function.
\end{proposition}
\begin{proof}
\cite{dangelo} or \cite{jabbari-egg-index}.
\end{proof}

Theorem \ref{theorem1} follows immediately from:

\begin{proposition}\label{proposition-1}
For each coordinate function $f=z_j$, let $M_{f}:L^2_{a}(\Om_1)\ra L^2_{a}(\Om_1)$ be the multiplication by $f$.
Then:

(a)
Given $i\in\bN^{m}$, we have
	\begin{align}
	\left[M_{z_1},M_{z_1}^*\right]\left(b_i\right)
	=\lambda b_i,\label{self-1}\quad
	\end{align}
	\begin{equation}
	\sqrt{\left[M_{z_2},M_{z_1}^*\right]\left[M_{z_2},M_{z_1}^*\right]^*}\left(b_i\right)
	=\mu b_i,\label{cross-1}
	\end{equation}
where
	\[
	\lambda
	=\frac{\omega(i_1,i_2,\ldots,i_m)}{\omega(i_1-1,i_2,\ldots, i_m)}\delta(i_1)
	-\frac{\omega(i_1+1,i_2,\ldots,i_m)}{\omega(i_1,i_2,\ldots, i_m)},
	\]
\[
\mu
=
\Bigg|\sqrt{\frac{\om(i_1,\ldots,i_m)\omega(i_1+1,i_2-1,i_3,\ldots,i_m)}{\omega(i_1,i_2-1,i_3,\ldots, i_m)^2}}-
\sqrt{\frac{\omega(i_1+1,i_2,\ldots,i_m)^2}{\omega(i_1,\ldots,i_m)\omega(i_1+1,i_2-1,i_3,\ldots, i_m)}}\Bigg|\delta(i_2),
\]
and
	\[
	\delta(x)=
	\begin{cases}
	0,&\text{if}\ x=0,\\
	1,&\text{otherwise}.
	\end{cases}
	\]
	
	(b)
	$\left[M_{z_1},M_{z_1}^*\right]$ is $p$-summable if and only if 
	\begin{equation*}
	p>
	\begin{cases}
	\frac{1}{2},&\text{if}\ \dim\Om_1=1,\\
	\max\left\{\dim\Om_1,p_1(\dim\Om_1-1)\right\},&\text{if}\ \dim\Om_1>1,
	\end{cases}\label{condition1}
	\end{equation*}
	where $\dim\Om_1=m$ is the complex dimension of $\Om_1$.
	
	(c)
	Assume $\dim\Om_1>1$.
	Then $\left[M_{z_2},M_{z_1}^*\right]$ is $p$-summable if and only if $p>\dim\Om_1$.
\end{proposition}

%\begin{equation}
%p>\dim\Om.\label{condition3}
%\end{equation}
%\end{proposition}

\begin{proof}
	(a)
	Straightforward computations.
	
	(b)
	We can write
	\[
	\lambda
	=\lambda'(\lambda''-1),
	\]
	which, by Lemma \ref{lemma2},
	\[
	\lambda'=
	\frac{\Gamma\left(\frac{i_1+2}{p_1}\right)}{\Gamma\left(\frac{i_1+1}{p_1}\right)}
	\frac{\Gamma\left(\frac{i_1+1}{p_1}+N\right)}{\Gamma\left(\frac{i_1+2}{p_1}+N\right)}\frac{\frac{i_1+1}{p_1}+N}{\frac{i_1+2}{p_1}+N},
	\]
	\[
	\lambda''
	=
	\frac{\Gamma\left(\frac{i_1+1}{p_1}\right)^2}{\Gamma\left(\frac{i_1}{p_1}\right)\Gamma\left(\frac{i_1+2}{p_1}\right)}
	\frac{\Gamma\left(\frac{i_1}{p_1}+N\right)\Gamma\left(\frac{i_1+2}{p_1}+N\right)}{\Gamma\left(\frac{i_1+1}{p_1}+N\right)^2}
	\frac{\left(\frac{i_1}{p_1}+N\right)\left(\frac{i_1+2}{p_1}+N\right)}{\left(\frac{i_1+1}{p_1}+N\right)^2}\delta(i_1),
	\]
	where
	\[
	N:=\sum_{j=2}^m\frac{i_j+1}{p_j}.
	\]

Formula (\ref{self-1}) shows that $\left[M_{z_1},M_{z_1}^*\right]$ is $p$-summable if and only if the infinite series $\sum_{i\in\bN^m} |\lambda|^p$ converges.
In the following, we derive the asymptotic formula for $\lambda$ when at least one of $i_1, N$ tends to infinity.

\textit{First, assume $m>1$.}
By Lemma \ref{fact-gamma},
\begin{equation}
\lambda'
\approx 
\begin{cases}
i_1^{\frac{1}{p_1}}/
\left(N+i_1\right)^{\frac{1}{p_1}},&\text{if} \ i_1>0,\\
N^{-\frac{1}{p_1}},&\text{if} \ i_1=0,
\end{cases}\label{prime-1}
\end{equation}
and 
\begin{equation}
\lambda''-1\approx 
\begin{cases}
-\frac{1}{i_1}+\frac{1}{i_1+N}-\frac{1}{(i_1+N)^2}
\approx \frac{N}{i_1(i_1+N)},&\text{if} \ i_1>0,\\
1,&\text{if} \ i_1=0.
\end{cases}\label{primeprime-1}
\end{equation}
Here, $f(x)\approx g(x)$ means that as $x$ grows large, the dominant term in the asymptotic expansion of $f(x)$ is $g(x)$, namely $f(x)=Cg(x)+O(x^{-k})$ for some $k\in\bN_+$ and $C>0$.
Note that to obtain the asymptotic formula (\ref{prime-1}) (respectively, (\ref{primeprime-1})), we have used Lemma \ref{fact-gamma} up to $O(x^{-1})$ (respectively, $O(x^{-2})$).
We have shown that
\[
\lambda
\approx
\begin{cases}
i_1^{\frac{1}{p_1}-1}
N/
\left(N+i_1\right)^{\frac{1}{p_1}+1},&\text{if} \ i_1>0,\\
N^{-\frac{1}{p_1}},&\text{if} \ i_1=0.
\end{cases}
\]
Therefore, $\sum|\lambda|^p$ converges if and only if both of the following series converge:
\[
\sum i_1^{\left(\frac{1}{p_1}-1\right)p}N^p/(N+i_1)^{\left(\frac{1}{p_1}+1\right)p},\quad\quad
\sum N^{-\frac{p}{p_1}}.
\]
According to Lemma \ref{fact-1}.(b), the first series converges exactly when $p>\max\{(p_1-1)m,m\}$, and the second converges exactly when $p>p_1(m-1)$.

\textit{Next, assume $m=1$.}
Then $N=0$, and by Lemma \ref{fact-gamma},
\[
\lambda
=\frac{i_1+1}{i_1+2}
\left(\frac{i_1(i_1+2)}{(i_1+1)^2}-1\right)
\approx\frac{1}{i_1^2}.
\]
Therefore, $\sum|\lambda|^p$ converges exactly when $p>1/2$.

The whole analysis proves (b).

(c)
Note that an operator $T$ on a Hilbert space is $p$-summable if and only if $\sqrt{TT^*}$ is so.
Formula (\ref{cross-1}) shows that $\left[M_{z_2},M_{z_1}^*\right]$ is $p$-summable if and only if the infinite series $\sum_{i\in\bN^m}\mu^p$ converges.
Since the region $i_2=0$ has no effect on the summability of this series, we assume $i_2>0$ from now on.
We can write
\[
\mu
=\mu'|\mu''-1|,
\]
\[
\mu'=\mu'_1\mu'_2,
\]
which, by Lemma \ref{lemma2},
\[
\mu'_1
=\sqrt{\frac{\Gamma\left(\frac{i_1+2}{p_1}\right)\Gamma\left(\frac{i_2+1}{p_2}\right)}{\Gamma\left(\frac{i_1+1}{p_1}\right)\Gamma\left(\frac{i_2}{p_2}\right)}
\frac{\Gamma\left(\frac{i_1+1}{p_1}+\frac{i_2+1}{p_2}+M\right)\Gamma\left(\frac{i_1+2}{p_1}+\frac{i_2}{p_2}+M\right)}{\Gamma\left(\frac{i_1+2}{p_1}+\frac{i_2+1}{p_2}+M\right)^2}},
\]
\[
\mu'_2
=\sqrt{\frac{\left(\frac{i_1+1}{p_1}+\frac{i_2+1}{p_2}+M\right)\left(\frac{i_1+2}{p_1}+\frac{i_2}{p_2}+M\right)}{\left(\frac{i_1+2}{p_1}+\frac{i_2+1}{p_2}+M\right)^2}},
\]
\[
\mu''=
\frac{\Gamma\left(\frac{i_1+1}{p_1}+\frac{i_2}{p_2}+M\right)\Gamma\left(\frac{i_1+2}{p_1}+\frac{i_2+1}{p_2}+M\right)}{\Gamma\left(\frac{i_1+1}{p_1}+\frac{i_2+1}{p_2}+M\right)\Gamma\left(\frac{i_1+2}{p_1}+\frac{i_2}{p_2}+M\right)}
\frac{\left(\frac{i_1+1}{p_1}+\frac{i_2}{p_2}+M\right)\left(\frac{i_1+2}{p_1}+\frac{i_2+1}{p_2}+M\right)}{\left(\frac{i_1+1}{p_1}+\frac{i_2+1}{p_2}+M\right)\left(\frac{i_1+2}{p_1}+\frac{i_2}{p_2}+M\right)},
\]
where
\[
M:=\sum_{j=3}^m\frac{i_j+1}{p_j}.
\]
By Lemma \ref{fact-gamma},
\[
\mu'\approx
i_1^{\frac{1}{2p_1}}i_2^{\frac{1}{2p_2}}/(i_1+i_2+M)^{\frac{1}{2}\left(\frac{1}{p_1}+\frac{1}{p_2}\right)},
\]
and 
\[
\mu''-1
\approx \frac{1}{i_1+i_2+M}-\frac{1}{(i_1+i_2+M)^2}
\approx \frac{1}{i_1+i_2+M},
\]
hence
\[
\mu
\approx
i_1^{\frac{1}{2p_1}}i_2^{\frac{1}{2p_2}}/(i_1+i_2+M)^{\frac{1}{2}\left(\frac{1}{p_1}+\frac{1}{p_2}\right)+1}.
\]
Therefore, $\sum \mu^p<\infty$ if and only if 
\[
\sum i_1^{\frac{p}{2p_1}}i_2^{\frac{p}{2p_2}}/(i_1+i_2+M)^{\frac{1}{2}\left(\frac{1}{p_1}+\frac{1}{p_2}\right)p+p}<\infty.
\]
According to Lemma \ref{fact-1}.(b), this happens exactly when $p>m$.
\end{proof}

\section{Proof of Theorem \ref{theorem2}}\label{section-theorem2}
This section proves Theorem \ref{theorem2} about the $p$-essential normality of the Bergman module $L^2_a(\Om_2)$ over the domain $\Om_2$ given in (\ref{egg2}).
For notational simplicity, we work on 
\begin{equation}
\Om_2:=\left\{\left(\sum\limits_{j=1}^{m}\left|z_j\right|^{2p_j}\right)^{a}+\left(\sum\limits_{k=1}^{n}\left|w_k\right|^{2q_k}\right)^{b}+\left(\sum\limits_{l=1}^{o}\left|u_l\right|^{2r_l}\right)^{c}+\cdots<1\right\}\sub\bC^{m+n+o+\cdots},\label{egg3}
\end{equation}
instead of (\ref{egg2}).
Similar to the discussions in Section \ref{section-theorem1}, the normalized monomials
\begin{equation}
b_{\al,\be,\ldots}:=\frac{z^\al w^{\be}\cdots}{\sqrt{\om_2(\al,\be,\ldots)}},\quad
(\al,\be,\ldots)\in\bN^{m+n+\cdots},\label{onb2}
\end{equation}where
\[
\om_2(\al,\be,\ldots):=\left\|z^\al w^{\be}\cdots\right\|^2_{L^2_a(\Om_2)},
\]
constitute an orthonormal basis for the Hilbert space $L^2_a(\Om_2)$.
An explicit formula for the norm of monomials is given by:

\begin{lemma}\label{lemma2}
	Given multi-indices $\al\in\bN^m, \be\in\bN^n, \ldots$, we have
	\[
	\om_2(\al,\be,\ldots)=
	\frac{\pi^{m+n+\cdots}}{\prod p_j\prod q_k\cdots}
	\frac{1}{ab\cdots}
	B\left(\frac{\al+1}{p}\right)
	B\left(\frac{\be+1}{q}\right)\cdots
	\frac{B\left(\left|\frac{\al+1}{ap}\right|,\left|\frac{\be+1}{bq}\right|,\ldots\right)}{\left|\frac{\al+1}{ap}\right|+\left|\frac{\be+1}{bq}\right|+\cdots}.
	\]
\end{lemma}
\begin{proof}
\cite{dangelo} or \cite{jabbari-egg-index}.
\end{proof}

Theorem \ref{theorem2} follows immediately from:

\begin{proposition}\label{proposition-2}
For each coordinate function $f=z_j,w_k,\ldots$, let $M_{f}:L^2_{a}(\Om_2)\ra L^2_{a}(\Om_2)$ be the multiplication by $f$.
%Consider the monomial orthonormal basis $b_{\al,\be,\gamma,\ldots}$ given by (\ref{onb2}).
Then:

(a)
Given $(\al,\be,\ldots)\in\bN^{m+n+\cdots}$, we have
\begin{align}
\left[M_{z_1},M_{z_1}^*\right](b_{\al,\be,\ldots})
=\lambda b_{\al,\be,\ldots},\label{self}\quad
\end{align}
\begin{equation}
\sqrt{\left[M_{z_2},M_{z_1}^*\right]\left[M_{z_2},M_{z_1}^*\right]^*}(b_{\al,\be,\ldots})
=\mu b_{\al,\be,\ldots},\label{cross1}
\end{equation}
\begin{equation}
\sqrt{\left[M_{z_1},M_{w_1}^*\right]\left[M_{z_1},M_{w_1}^*\right]^*}(b_{\al,\be,\ldots})
=\nu b_{\al,\be,\ldots},\label{cross2}
\end{equation}
where
\[
\lambda
=\frac{\om(\al,\be,\ldots)}{\om(\al_1-1,\al_2,\ldots,\al_m,\be,\ldots)}\delta(\al_1)
-\frac{\om(\al_1+1,\al_2,\ldots,\al_m,\be,\ldots)}{\om(\al,\be,\ldots)},
\]
\begin{multline*}
\mu
=
\Bigg|\sqrt{\frac{\om(\al,\be,\ldots)\om(\al_1+1,\al_2-1,\al_3,\ldots,\al_m,\be,\ldots)}{\om(\al_1,\al_2-1,\al_3,\ldots,\al_m,\be,\ldots)^2}}-\\
\sqrt{\frac{\om(\al_1+1,\al_2,\ldots,\al_m,\be,\ldots)^2}{\om(\al,\be,\ldots)\om(\al_1+1,\al_2-1,\al_3,\ldots, \al_m,\be,\ldots)}}\Bigg|\delta(\al_2),
\end{multline*}
\begin{multline*}
\nu
=
\Bigg|\sqrt{\frac{\om(\al,\be,\ldots)\om(\al_1+1,\al_2,\ldots,\al_m,\be_1-1,\be_2,\ldots,\be_n,\gamma,\ldots)}{\om(\al,\be_1-1,\be_2,\ldots,\be_n,\gamma,\ldots)^2}}-\\
\sqrt{\frac{\om(\al_1+1,\al_2,\ldots,\al_m,\be,\ldots)^2}{\om(\al,\be,\ldots)\om(\al_1+1,\al_2,\ldots,\al_m,\be_1-1,\be_2,\ldots,\be_n,\gamma,\ldots)}}\Bigg|\delta(\be_1),
\end{multline*}
and
\[
\delta(x)=
\begin{cases}
0,&\text{if}\ x=0,\\
1,&\text{otherwise}.
\end{cases}
\]

(b)
$\left[M_{z_1},M_{z_1}^*\right]$ is $p$-summable if and only if  
\begin{equation*}
p>
\begin{cases}
\frac{1}{2},&\text{if}\ \dim\Om_2=1,\\
\max\left\{\dim\Om_2,ap_1(\dim\Om_2-m)\right\},&\text{if}\ \dim\Om_2>1,m=1,\\
\max\left\{\dim\Om_2,p_1(\dim\Om_2-1),ap_1(\dim\Om_2-m)\right\},&\text{if}\ \dim\Om_2>1,m>1.
\end{cases}\label{condition1}
\end{equation*}

(c)
Assume $m>1$.
$\left[M_{z_2},M_{z_1}^*\right]$ is $p$-summable if and only if  
\begin{equation*}
p>
\begin{cases}
\max\left\{\dim\Om_2\right\},&\text{if}\ a= 1,\\
\max\left\{\dim\Om_2,\frac{2a(\dim\Om_2-m)}{1/p_1+1/p_2}\right\},&\text{if}\ a\neq 1.
\end{cases}\label{condition2}
\end{equation*}

(d)
$\left[M_{z_1},M_{w_1}^*\right]$ is $p$-summable if and only if $p>\dim\Om_2$.
\end{proposition}

%\begin{equation}
%p>\dim\Om.\label{condition3}
%\end{equation}
%\end{proposition}

\begin{proof}
(a)
Straightforward computations.

(b)
We can write
\[
\lambda
=\lambda'(\lambda''-1),
\]
\[
\lambda''
=\lambda''_1\lambda''_2\delta(\al_1),
\]
which, by Lemma \ref{lemma2},
\[
\lambda'=
\frac{\Gamma\left(\frac{\al_1+2}{p_1}\right)\Gamma\left(\frac{\al_1+1}{p_1}+A\right)}{\Gamma\left(\frac{\al_1+1}{p_1}\right)\Gamma\left(\frac{\al_1+2}{p_1}+A\right)}
\frac{\Gamma\left(\frac{\al_1+2}{ap_1}+\frac{A}{a}\right)\Gamma\left(\frac{\al_1+1}{ap_1}+\frac{A}{a}+L\right)}{\Gamma\left(\frac{\al_1+1}{ap_1}+\frac{A}{a}\right)\Gamma\left(\frac{\al_1+2}{ap_1}+\frac{A}{a}+L\right)}
\frac{\frac{\al_1+1}{ap_1}+\frac{A}{a}+L}{\frac{\al_1+2}{ap_1}+\frac{A}{a}+L},
\]
\[
\lambda''_1
=
\frac{\Gamma\left(\frac{\al_1+1}{p_1}\right)^2}{\Gamma\left(\frac{\al_1}{p_1}\right)\Gamma\left(\frac{\al_1+2}{p_1}\right)}
\frac{\Gamma\left(\frac{\al_1}{p_1}+A\right)\Gamma\left(\frac{\al_1+2}{p_1}+A\right)}{\Gamma\left(\frac{\al_1+1}{p_1}+A\right)^2}
\frac{\Gamma\left(\frac{\al_1+1}{ap_1}+\frac{A}{a}\right)^2}{\Gamma\left(\frac{\al_1}{ap_1}+\frac{A}{a}\right)\Gamma\left(\frac{\al_1+2}{ap_1}+\frac{A}{a}\right)},
\]
\[
\lambda''_2
=
\frac{\Gamma\left(\frac{\al_1}{ap_1}+\frac{A}{a}+L\right)\Gamma\left(\frac{\al_1+2}{ap_1}+\frac{A}{a}+L\right)}{\Gamma\left(\frac{\al_1+1}{ap_1}+\frac{A}{a}+L\right)^2}
\frac{\left(\frac{\al_1}{ap_1}+\frac{A}{a}+L\right)\left(\frac{\al_1+2}{ap_1}+\frac{A}{a}+L\right)}{\left(\frac{\al_1+1}{ap_1}+\frac{A}{a}+L\right)^2},
\]
where
\[
A=\sum_{j=2}^m\frac{\al_j+1}{p_j},
\]
\[
L=\sum_{k=1}^n\frac{\be_k+1}{bq_k}+\sum_{l=1}^o\frac{\gamma_l+1}{cr_l}+\cdots.
\]

Formula (\ref{self}) shows that $\left[M_{z_1},M_{z_1}^*\right]$ is $p$-summable if and only if the infinite series $\sum |\lambda|^p$ converges.
In the following, we derive the asymptotic formula for $\lambda$ when at least one of $\al_1, A, L$ tends to infinity.

\textit{First, assume $m>1$.}
By Lemma \ref{fact-gamma},
\begin{equation}
\lambda'
\approx 
\begin{cases}
\al_1^{\frac{1}{p_1}}
\left(\al_1+A\right)^{-\frac{1}{p_1}(1-\frac{1}{a})}
/\left(\al_1+A+L\right)^{\frac{1}{ap_1}},&\text{if}\ \al_1>0,\\
A^{-\frac{1}{p_1}(1-\frac{1}{a})}
/\left(A+L\right)^{\frac{1}{ap_1}},&\text{if}\ \al_1=0,
\end{cases}\label{prime}
\end{equation}
and
\begin{equation}
\lambda''-1
\approx 
-\frac{1}{\al_1}
+\frac{1}{\al_1+A}
-\frac{a^{-2}}{\al_1+A}
+\frac{a^{-2}}{\al_1+A+L}
\approx
\frac{A}{\al_1(\al_1+A)}+\frac{a^{-2}L}{(\al_1+A)(\al_1+A+L)},\label{primeprime}
\end{equation}
if $\al_1>0$, and $\lambda''-1= -1$ if $\al_1=0$.
As before, $f(x)\approx g(x)$ means that as $x$ grows large, the dominant term in the asymptotic expansion of $f(x)$ is $g(x)$, namely $f(x)=Cg(x)+O(x^{-k})$ for some $k\in\bN_+$ and $C>0$.
Note that to obtain the asymptotic formula (\ref{prime}) (respectively, (\ref{primeprime})), we have used Lemma \ref{fact-gamma} up to $O(x^{-1})$ (respectively, $O(x^{-2})$).
We have shown that
\[
\lambda
\approx
\begin{cases}
\frac{\al_1^{\frac{1}{p_1}-1}
	\left(\al_1+A\right)^{-\frac{1}{p_1}\left(1-\frac{1}{a}\right)-1}A}{\left(\al_1+A+L\right)^{\frac{1}{ap_1}}}
+
\frac{\al_1^{\frac{1}{p_1}}
	\left(\al_1+A\right)^{-\frac{1}{p_1}\left(1-\frac{1}{a}\right)-1}L}{\left(\al_1+A+L\right)^{\frac{1}{ap_1}+1}},&\text{if}\ \al_1>0,\\
A^{-\frac{1}{p_1}(1-\frac{1}{a})}
/\left(A+L\right)^{\frac{1}{ap_1}},&\text{if}\ \al_1=0.
\end{cases}
\]
According to Lemma \ref{fact-sumoftwoseries}, $\sum|\lambda|^p$ converges if and only if all of the following series converge:
\begin{equation}
\sum \al_1^{\left(\frac{1}{p_1}-1\right)p}(\al_1+A)^{-\left(\frac{1}{p_1}\left(1-\frac{1}{a}\right)+1\right)p}A^p/(\al_1+A+L)^{\frac{p}{ap_1}},\label{22-first}
\end{equation}
\begin{equation}
\sum \al_1^{\frac{p}{p_1}}(\al_1+A)^{-\left(\frac{1}{p_1}\left(1-\frac{1}{a}\right)+1\right)p}L^p/(\al_1+A+L)^{\left(\frac{1}{ap_1}+1\right)p},\label{2-second}
\end{equation}
\begin{equation}
\sum A^{-\frac{p}{p_1}\left(1-\frac{1}{a}\right)}/(A+L)^{\frac{p}{ap_1}}.\label{2-third}
\end{equation}
Similar arguments as in the proof of Lemma \ref{fact-1}.(b) show that the convergence of the series (\ref{22-first}) is equivalent to the convergence of
\[
\sum_{i\in\bN^{3}} \frac{i_1^{\left(\frac{1}{p_1}-1\right)p}(i_1+i_2)^{-\left(\frac{1}{p_1}\left(1-\frac{1}{a}\right)+1\right)p}i_2^{p+m-2}i_3^{m'-1}}{(i_1+i_2+i_3)^{\frac{p}{ap_1}}},\label{2-first}
\]
where 
\[
m':=\dim\Om_2-m=n+o+\cdots.
\]
According to Lemma \ref{fact-3}.(a), this latter series converges exactly when 
\begin{equation*}
p>\max\left\{m+m',p_1\left(m+m'-1\right),ap_1m'\right\}.
\end{equation*}
Likewise, the series (\ref{2-second}) and (\ref{2-third}) converge exactly when 
\begin{equation*}
p>\max\left\{m+m',ap_1m'\right\}\quad\text{and}\quad
p>\max\left\{p_1\left(m+m'-1\right),ap_1m'\right\},
\end{equation*}
respectively.

\textit{Next, assume $m=1$.}
Then, $A=0$ and 
\[
\lambda'
=\frac{\Gamma\left(\frac{\al_1+2}{ap_1}\right)\Gamma\left(\frac{\al_1+1}{ap_1}+L\right)}{\Gamma\left(\frac{\al_1+1}{ap_1}\right)\Gamma\left(\frac{\al_1+2}{ap_1}+L\right)}
\frac{\frac{\al_1+1}{ap_1}+L}{\frac{\al_1+2}{ap_1}+L},
\]
\[
\lambda''=
\begin{cases}
\frac{\Gamma\left(\frac{\al_1+1}{ap_1}\right)^2}{\Gamma\left(\frac{\al_1}{ap_1}\right)\Gamma\left(\frac{\al_1+2}{ap_1}\right)}
\frac{\Gamma\left(\frac{\al_1}{ap_1}+L\right)\Gamma\left(\frac{\al_1+2}{ap_1}+L\right)}{\Gamma\left(\frac{\al_1+1}{ap_1}+L\right)^2}
\frac{\left(\frac{\al_1}{ap_1}+L\right)\left(\frac{\al_1+2}{ap_1}+L\right)}{\left(\frac{\al_1+1}{ap_1}+L\right)^2}
,&\text{if}\ \al_1>0,m'>0,\\
\frac{\al_1(\al_1+2)}{(\al_1+1)^2}
,&\text{if}\ \al_1>0,m'=0,\\
0,&\text{if}\ \al_1=0.
\end{cases}
\]

(Note that $L=0$ if $m'=0$.)
By Lemma \ref{fact-gamma},
\[
\lambda'
\approx 
\begin{cases}
\al_1^{\frac{1}{ap_1}}/
\left(\al_1+L\right)^{\frac{1}{ap_1}},&\text{if}\ \al_1>0,\\
L^{-\frac{1}{ap_1}},&\text{if}\ \al_1=0,
\end{cases}
\]
\[
\lambda''-1
\approx
\begin{cases}
-\frac{1}{\al_1}
+\frac{1}{\al_1+L}
-\frac{1}{(\al_1+L)^2}
\approx
\frac{L}{\al_1(\al_1+L)},&\text{if}\ \al_1>0,m'>0,\\
1/\al_1^2,&\text{if}\ \al_1>0,m'=0,\\
1,&\text{if}\ \al_1=0.
\end{cases}
\]
We have shown that
\[
\lambda
\approx
\begin{cases}
\al_1^{\frac{1}{ap_1}-1}L/\left(\al_1+L\right)^{\frac{1}{ap_1}+1},&\text{if}\ \al_1>0,m'>0,\\
1/\al_1^2,&\text{if}\ \al_1>0,m'=0,\\
L^{-\frac{1}{ap_1}},&\text{if}\ \al_1=0.
\end{cases}
\]
According to Lemma \ref{fact-1}.(b), the series $\sum|\lambda|^p$ converges exactly when 
\[
p>
\begin{cases}
\max\left\{ap_1m',m'+1\right\},&\text{if}\ m'>0,\\
\frac{1}{2},&\text{if}\ m'=0.
\end{cases}
\]

The whole analysis proves (b).

%According to Lemma \ref{fact}, this happens exactly when
%\[
%p>\begin{cases}
%p_1(m-1),& p_1\geq 1,\\
%\frac{m}{2},& p_1\leq \frac{1}{2}\ \text{or}\  \frac{1}{2}\leq p_1<1.
%\end{cases}
%\]
%(This follows from Lemma \ref{fact} when $p_1\neq 1$.
%Now assume $p_1=1$. 
%Our series is $S:=\sum i_1^{-p}(N+i_1)^{-p}$.
%If $p<1$, then $S$ diverges by Lemma \ref{fact}.
%If $p=1$, by the integral test, $S$ diverges.
%If $p>1$, then $S$ converges exactly when $p>m-1$, according to Lemma \ref{fact}.)

(c)
Formula (\ref{cross1}) shows that $\left[M_{z_2},M_{z_1}^*\right]$ is $p$-summable if and only if the infinite series $\sum\mu^p$ converges.
Since the region $\al_2=0$ has no effect on the summability of this series, we assume $\al_2>0$.
We can write
\[
\mu
=\mu'\left|\mu''-1\right|,
\]
\[
\mu'=\mu'_1\mu'_2\mu'_3\mu'_4,
\]
\[
\mu''=\mu''_1\mu''_2\mu''_3\mu''_4,
\]
which, by Lemma \ref{lemma2},
\[
\mu'_1
=\sqrt{\frac{\Gamma\left(\frac{\al_1+2}{p_1}\right)\Gamma\left(\frac{\al_2+1}{p_2}\right)}{\Gamma\left(\frac{\al_1+1}{p_1}\right)\Gamma\left(\frac{\al_2}{p_2}\right)}
\frac{\Gamma\left(\frac{\al_1+1}{p_1}+\frac{\al_2+1}{p_2}+\cA\right)\Gamma\left(\frac{\al_1+2}{p_1}+\frac{\al_2}{p_2}+\cA\right)}{\Gamma\left(\frac{\al_1+2}{p_1}+\frac{\al_2+1}{p_2}+\cA\right)^2}},
\]
\[
\mu'_2
=\sqrt{\frac{\Gamma\left(\frac{\al_1+2}{ap_1}+\frac{\al_2+1}{ap_2}+\frac{\cA}{a}\right)^2}{\Gamma\left(\frac{\al_1+1}{ap_1}+\frac{\al_2+1}{ap_2}+\frac{\cA}{a}\right)\Gamma\left(\frac{\al_1+2}{ap_1}+\frac{\al_2}{ap_2}+\frac{\cA}{a}\right)}},
\]
\[
\mu'_3
=\sqrt{\frac{\Gamma\left(\frac{\al_1+1}{ap_1}+\frac{\al_2+1}{ap_2}+\frac{\cA}{a}+L\right)\Gamma\left(\frac{\al_1+2}{ap_1}+\frac{\al_2}{ap_2}+\frac{\cA}{a}+L\right)}{\Gamma\left(\frac{\al_1+2}{ap_1}+\frac{\al_2+1}{ap_2}+\frac{\cA}{a}+L\right)^2}},
\]
\[
\mu'_4
=\frac{\left(\frac{\al_1+1}{ap_1}+\frac{\al_2+1}{ap_2}+\frac{\cA}{a}+L\right)\left(\frac{\al_1+2}{ap_1}+\frac{\al_2}{ap_2}+\frac{\cA}{a}+L\right)}{\left(\frac{\al_1+2}{ap_1}+\frac{\al_2+1}{ap_2}+\frac{\cA}{a}+L\right)^2},
\]

\[
\mu''_1
=\frac{\Gamma\left(\frac{\al_1+1}{p_1}+\frac{\al_2}{p_2}+\cA\right)\Gamma\left(\frac{\al_1+2}{p_1}+\frac{\al_2+1}{p_2}+\cA\right)}{\Gamma\left(\frac{\al_1+1}{p_1}+\frac{\al_2+1}{p_2}+\cA\right)\Gamma\left(\frac{\al_1+2}{p_1}+\frac{\al_2}{p_2}+\cA\right)},
\]
\[
\mu''_2
=\frac{\Gamma\left(\frac{\al_1+1}{ap_1}+\frac{\al_2+1}{ap_2}+\frac{\cA}{a}\right)\Gamma\left(\frac{\al_1+2}{ap_1}+\frac{\al_2}{ap_2}+\frac{\cA}{a}\right)}{\Gamma\left(\frac{\al_1+1}{ap_1}+\frac{\al_2}{ap_2}+\frac{\cA}{a}\right)\Gamma\left(\frac{\al_1+2}{ap_1}+\frac{\al_2+1}{ap_2}+\frac{\cA}{a}\right)},
\]
\[
\mu''_3
=\frac{\Gamma\left(\frac{\al_1+1}{ap_1}+\frac{\al_2}{ap_2}+\frac{\cA}{a}+L\right)\Gamma\left(\frac{\al_1+2}{ap_1}+\frac{\al_2+1}{ap_2}+\frac{\cA}{a}+L\right)}{\Gamma\left(\frac{\al_1+1}{ap_1}+\frac{\al_2+1}{ap_2}+\frac{\cA}{a}+L\right)\Gamma\left(\frac{\al_1+2}{ap_1}+\frac{\al_2}{ap_2}+\frac{\cA}{a}+L\right)},
\]
\[
\mu''_4
=\frac{\left(\frac{\al_1+1}{ap_1}+\frac{\al_2}{ap_2}+\frac{\cA}{a}+L\right)\left(\frac{\al_1+2}{ap_1}+\frac{\al_2+1}{ap_2}+\frac{\cA}{a}+L\right)}{\left(\frac{\al_1+1}{ap_1}+\frac{\al_2+1}{ap_2}+\frac{\cA}{a}+L\right)\left(\frac{\al_1+2}{ap_1}+\frac{\al_2}{ap_2}+\frac{\cA}{a}+L\right)},
\]
where
\[
\cA=\sum_{j=3}^m\frac{i_j+1}{p_j},
\]
\[
L=\sum_{k=1}^n\frac{\be_k+1}{bq_k}+\sum_{l=1}^o\frac{\gamma_l+1}{cr_l}+\cdots.
\]
By Lemma \ref{fact-gamma},
\[
\mu'\approx
\frac{\al_1^{\frac{1}{2p_1}}\al_2^{\frac{1}{2p_2}}(\al_1+\al_2+\cA)^{-\frac{1}{2}\left(\frac{1}{p_1}+\frac{1}{p_2}\right)\left(1-\frac{1}{a}\right)}}{(\al_1+\al_2+\cA+L)^{\frac{1}{2a}\left(\frac{1}{p_1}+\frac{1}{p_2}\right)}},
\]
\begin{multline*}
\mu''-1
\approx 
\frac{1}{\al_1+\al_2+\cA}
-\frac{a^{-2}}{\al_1+\al_2+\cA}
+\frac{a^{-2}}{\al_1+\al_2+\cA+L}
-\frac{a^{-2}}{(\al_1+\al_2+\cA+L)^2}\\
\approx\frac{a^{2}-1}{\al_1+\al_2+\cA}
+\frac{1}{\al_1+\al_2+\cA+L}.
\end{multline*}
Hence,
\[
\mu
\approx
\begin{cases}
\frac{\al_1^{\frac{1}{2p_1}}\al_2^{\frac{1}{2p_2}}(\al_1+\al_2+\cA)^{-\frac{1}{2}\left(\frac{1}{p_1}+\frac{1}{p_2}\right)\left(1-\frac{1}{a}\right)-1}\left|\al_1+\al_2+\cA-tL\right|}{(\al_1+\al_2+\cA+L)^{\frac{1}{2a}\left(\frac{1}{p_1}+\frac{1}{p_2}\right)+1}},&\text{if}\ a<1,\\
\al_1^{\frac{1}{2p_1}}\al_2^{\frac{1}{2p_2}}/(\al_1+\al_2+\cA+L)^{\frac{1}{2}\left(\frac{1}{p_1}+\frac{1}{p_2}\right)+1},&\text{if}\ a=1,\\
\frac{\al_1^{\frac{1}{2p_1}}\al_2^{\frac{1}{2p_2}}(\al_1+\al_2+\cA)^{-\frac{1}{2}\left(\frac{1}{p_1}+\frac{1}{p_2}\right)\left(1-\frac{1}{a}\right)-1}}{(\al_1+\al_2+\cA+L)^{\frac{1}{2a}\left(\frac{1}{p_1}+\frac{1}{p_2}\right)}},&\text{if}\ a>1,
\end{cases}
\]
where
\[
t:=a^{-2}-1>0.
\]
Therefore, $\sum\mu^p$ converges if and only if the following series converges: 
\[
\begin{cases}
\sum\frac{\al_1^{\frac{p}{2p_1}}\al_2^{\frac{p}{2p_2}}(\al_1+\al_2+\cA)^{-\frac{p}{2}\left(\frac{1}{p_1}+\frac{1}{p_2}\right)\left(1-\frac{1}{a}\right)-p}\left|\al_1+\al_2+\cA-tL\right|^p}{(\al_1+\al_2+\cA+L)^{\frac{p}{2a}\left(\frac{1}{p_1}+\frac{1}{p_2}\right)+p}},&\text{if}\ a<1,\\
\sum \al_1^{\frac{p}{2p_1}}\al_2^{\frac{p}{2p_2}}/(\al_1+\al_2+\cA+L)^{\frac{p}{2}\left(\frac{1}{p_1}+\frac{1}{p_2}\right)+p},&\text{if}\ a=1,\\
\sum\frac{\al_1^{\frac{p}{2p_1}}\al_2^{\frac{p}{2p_2}}(\al_1+\al_2+\cA)^{-\frac{p}{2}\left(\frac{1}{p_1}+\frac{1}{p_2}\right)\left(1-\frac{1}{a}\right)-p}}{(\al_1+\al_2+\cA+L)^{\frac{p}{2a}\left(\frac{1}{p_1}+\frac{1}{p_2}\right)}},&\text{if}\ a>1.
\end{cases}\label{condition2}
\]
Similar arguments as in the proof of Lemma \ref{fact-1}.(b) show that the convergence of this series is equivalent to the convergence of
\[
\begin{cases}
\sum_{i\in\bN^{4}}\frac{i_1^{\frac{p}{2p_1}}i_2^{\frac{p}{2p_2}}(i_1+i_2+i_3)^{-\frac{p}{2}\left(\frac{1}{p_1}+\frac{1}{p_2}\right)\left(1-\frac{1}{a}\right)-p}\left|i_1+i_2+i_3-i_4\right|^pi_3^{m-3}i_4^{m'-1}}{(i_1+i_2+i_3+i_4)^{\frac{p}{2a}\left(\frac{1}{p_1}+\frac{1}{p_2}\right)+p}},&\text{if}\ a<1,\\
\sum_{i\in\bN^{4}} i_1^{\frac{p}{2p_1}}i_2^{\frac{p}{2p_2}}i_3^{m-3}i_4^{m'-1}/(i_1+i_2+i_3+i_4)^{\frac{p}{2}\left(\frac{1}{p_1}+\frac{1}{p_2}\right)+p},&\text{if}\ a=1,\\
\sum_{i\in\bN^{4}}\frac{i_1^{\frac{p}{2p_1}}i_2^{\frac{p}{2p_2}}(i_1+i_2+i_3)^{-\frac{p}{2}\left(\frac{1}{p_1}+\frac{1}{p_2}\right)\left(1-\frac{1}{a}\right)-p}i_3^{m-3}i_4^{m'-1}}{(i_1+i_2+i_3+i_4)^{\frac{p}{2a}\left(\frac{1}{p_1}+\frac{1}{p_2}\right)}},&\text{if}\ a>1.
\end{cases}\label{condition2}
\]
According to Lemmas \ref{fact-1} and \ref{fact-3}.(c, d), this latter series converges exactly when 
\begin{equation}
p>
\begin{cases}
\max\left\{\dim\Om_2\right\},&\text{if}\ a= 1,\\
\max\left\{\dim\Om_2,\frac{2a(\dim\Om_2-m)}{p_1^{-1}+p_2^{-1}}\right\},&\text{if}\ a\neq 1.
\end{cases}\label{condition2}
\end{equation}

(d)
Formula (\ref{cross2}) shows that $\left[M_{z_1},M_{w_1}^*\right]$ is $p$-summable if and only if the infinite series $\sum\nu^p$ converges.
Since the region $\be_1=0$ has no effect on the summability of this series, we assume $\be_1>0$.
We can write
\[
\nu
=\nu'\left|\nu''-1\right|,
\]
\[
\nu'=\nu'_1\nu'_2\nu'_3,
\]
\[
\nu''=\nu''_1\nu''_2,
\]
which, by Lemma \ref{lemma2},
\[
\nu'_1
=\sqrt{\frac{\Gamma\left(\frac{\al_1+2}{p_1}\right)\Gamma\left(\frac{\be_1+1}{q_1}\right)}{\Gamma\left(\frac{\al_1+1}{p_1}\right)\Gamma\left(\frac{\be_1}{q_1}\right)}
\frac{\Gamma\left(\frac{\al_1+1}{p_1}+A\right)\Gamma\left(\frac{\be_1}{q_1}+B\right)}{\Gamma\left(\frac{\al_1+2}{p_1}+A\right)\Gamma\left(\frac{\be_1+1}{q_1}+B\right)}\frac{\Gamma\left(\frac{\al_1+2}{ap_1}+\frac{A}{a}\right)\Gamma\left(\frac{\be_1+1}{bq_1}+\frac{B}{b}\right)}{\Gamma\left(\frac{\al_1+1}{ap_1}+\frac{A}{a}\right)\Gamma\left(\frac{\be_1}{bq_1}+\frac{B}{b}\right)}},
\]
\[
\nu'_2
=\sqrt{\frac{\Gamma\left(\frac{\al_1+1}{ap_1}+\frac{\be_1+1}{bq_1}+\frac{A}{a}+\frac{B}{b}+\cL\right)\Gamma\left(\frac{\al_1+2}{ap_1}+\frac{\be_1}{bq_1}+\frac{A}{a}+\frac{B}{b}+\cL\right)}{\Gamma\left(\frac{\al_1+2}{ap_1}+\frac{\be_1+1}{bq_1}+\frac{A}{a}+\frac{B}{b}+\cL\right)^2}},
\]
\[
\nu'_3
=\sqrt{\frac{\left(\frac{\al_1+1}{ap_1}+\frac{\be_1+1}{bq_1}+\frac{A}{a}+\frac{B}{b}+\cL\right)\left(\frac{\al_1+2}{ap_1}+\frac{\be_1}{bq_1}+\frac{A}{a}+\frac{B}{b}+\cL\right)}{\left(\frac{\al_1+2}{ap_1}+\frac{\be_1+1}{bq_1}+\frac{A}{a}+\frac{B}{b}+\cL\right)^2}},
\]
\[
\nu''_1
=\frac{\Gamma\left(\frac{\al_1+1}{ap_1}+\frac{\be_1}{bq_1}+\frac{A}{a}+\frac{B}{b}+\cL\right)\Gamma\left(\frac{\al_1+2}{ap_1}+\frac{\be_1+1}{bq_1}+\frac{A}{a}+\frac{B}{b}+\cL\right)}{\Gamma\left(\frac{\al_1+1}{ap_1}+\frac{\be_1+1}{bq_1}+\frac{A}{a}+\frac{B}{b}+\cL\right)\Gamma\left(\frac{\al_1+2}{ap_1}+\frac{\be_1}{bq_1}+\frac{A}{a}+\frac{B}{b}+\cL\right)},
\]
\[
\nu''_2
=\frac{\left(\frac{\al_1+1}{ap_1}+\frac{\be_1}{bq_1}+\frac{A}{a}+\frac{B}{b}+\cL\right)\left(\frac{\al_1+2}{ap_1}+\frac{\be_1+1}{bq_1}+\frac{A}{a}+\frac{B}{b}+\cL\right)}{\left(\frac{\al_1+1}{ap_1}+\frac{\be_1+1}{bq_1}+\frac{A}{a}+\frac{B}{b}+\cL\right)\left(\frac{\al_1+2}{ap_1}+\frac{\be_1}{bq_1}+\frac{A}{a}+\frac{B}{b}+\cL\right)},
\]
where
\[
A=\sum_{j=2}^m\frac{\al_j+1}{p_j},
\]
\[
B=\sum_{k=2}^n\frac{\be_k+1}{q_k},
\]
\[
\cL=\sum_{l=1}^o\frac{\gamma_l+1}{cr_l}+\cdots.
\]
By Lemma \ref{fact-gamma},
\[
\nu'\approx
\al_1^{\frac{1}{2p_1}}
\be_1^{\frac{1}{2q_1}}
(\al_1+A)^{-\frac{1}{2p_1}\left(1-\frac{1}{a}\right)}
(\be_1+B)^{-\frac{1}{2q_1}\left(1-\frac{1}{b}\right)}/(\al_1+\be_1+A+B+\cL)^{\frac{1}{2}\left(\frac{1}{ap_1}+\frac{1}{bq_1}\right)},
\]
\[
\nu''-1
\approx 
\frac{1}{\al_1+\be_1+A+B+\cL}
-\frac{1}{(\al_1+\be_1+A+B+\cL)^2}
\approx
\frac{1}{\al_1+\be_1+A+B+\cL},
\]
so
\[
\nu\approx
\al_1^{\frac{1}{2p_1}}
\be_1^{\frac{1}{2q_1}}
(\al_1+A)^{-\frac{1}{2p_1}\left(1-\frac{1}{a}\right)}
(\be_1+B)^{-\frac{1}{2q_1}\left(1-\frac{1}{b}\right)}/(\al_1+\be_1+A+B+\cL)^{\frac{1}{2}\left(\frac{1}{ap_1}+\frac{1}{bq_1}\right)+1}.
\]
Therefore, $\sum\nu^p$ converges if and only if the following series converges:
\[
\sum
 \frac{\al_1^{\frac{p}{2p_1}}
 \be_1^{\frac{p}{2q_1}}
	(\al_1+A)^{-\frac{p}{2p_1}\left(1-\frac{1}{a}\right)}
	(\be_1+B)^{-\frac{p}{2q_1}\left(1-\frac{1}{b}\right)}}{(\al_1+\be_1+A+B+\cL)^{\frac{p}{2}\left(\frac{1}{ap_1}+\frac{1}{bq_1}\right)+p}}.
\]
Similar arguments as in the proof of Lemma \ref{fact-1}.(b) show that the convergence of this series is equivalent to 
\[
\sum_{i\in\bN^{5}}
\frac{i_1^{\frac{p}{2p_1}}i_3^{\frac{p}{2q_1}}(i_1+i_2)^{-\frac{p}{2p_1}\left(1-\frac{1}{a}\right)}(i_3+i_4)^{-\frac{p}{2q_1}\left(1-\frac{1}{b}\right)}i_2^{m-2}i_4^{n-2}i_5^{m'-1}}{(i_1+i_2+i_3+i_4+i_5)^{\frac{p}{2}\left(\frac{1}{ap_1}+\frac{1}{bq_1}\right)+p}}.
\]
According to Lemma \ref{fact-3}.(f), this series converges exactly when $p>\dim\Om_2$.
\end{proof}

%According to Lemma \ref{fact}, this happens exactly when
%\begin{equation}
%p>\dim\Om.\label{condition3}
%\end{equation}
%
%Theorem \ref{theorem1} is proved by putting three conditions (\ref{condition1}), (\ref{condition2}) and (\ref{condition3}) together.
%
%\end{proof}

\noindent
\textbf{Acknowledgments.}
The author would like to thank Richard Rochberg and Xiang Tang for reading a preliminary version of this paper and making helpful suggestions.

\noindent
\textbf{Data availability.}
Data sharing not applicable to this article as no datasets were generated or analysed during the current study.

%\section*{Acknowledgments}
%\noindent
%The author would like to thank Richard Rochberg and Xiang Tang for reading a preliminary version of this paper and making helpful suggestions.

\end{document}